\documentclass[review,nopreprintline]{elsarticle}

	\usepackage{amssymb,amsmath,bbm,latexsym,amscd,tikz-cd,enumerate}
	\usepackage{amsthm}
	\usetikzlibrary{cd}
	\usepackage{changes}
	\usepackage{tcolorbox}
	\usepackage{mdframed}
	\usepackage{lipsum}
	\usepackage{mathtools}
	\usepackage[thinc]{esdiff} 

	\newmdtheoremenv{theo}{Theorem}
	
   	\newtheorem{theorem}{Theorem}[section]
    \newtheorem{corollary}[theorem]{Corollary}
    
    \newtheorem{proposition}[theorem]{Proposition}
	\theoremstyle{definition}

    \numberwithin{equation}{section}
	\theoremstyle{remark}
    \newtheorem{remark}[equation]{Remark}

	\newcommand{\fg}{\mathfrak{g}}

	\newcommand{\ot}{\otimes}
	
	\newcommand{\Z}{\mathbb{Z}}
	\newcommand{\Cx}{\mathbb{C}}

	\newcommand{\Gl}{{\mathcal{G}}}

\begin{document}
	\begin{frontmatter}
	
	    \title{Superelliptic Affine Lie algebras and orthogonal polynomials}
	    
		    \author{Felipe Albino dos Santos}
		    \author{Mikhail Neklyudov}
		    \author{Vyacheslav Futorny }
	  
	    \begin{abstract}
	    	We construct two families of orthogonal polynomials associated with the universal central extensions of the superelliptic Lie algebras. These polynomials satisfy certain fourth order linear differential equations,  and one of the families is a particular collection of associated ultraspherical polynomials. We show that the generating functions of the polynomials satisfy fourth order linear PDEs. Since these generating functions  can be represented by superelliptic integrals, we have examples of linear PDEs of fourth order with explicit solutions without  complete integrability.
	    \end{abstract}
	
	
	\end{frontmatter}

	\section*{Introduction}
	    Let $\fg$ be a simple finite-dimensional complex Lie algebra and $\fg\ot\Cx[t,t^{-1}]$ the loop algebra of $\fg$ with commutation relations $[x\ot f, y\ot g]=[x,y]\ot fg$, for $x,y\in\fg$ and $f,g\in\Cx[t,t^{-1}]$.  In the construction of the loop algebra we may replace the Laurent polynomial algebra $\Cx[t,t^{-1}]$ by any other commutative associative complex algebra $R$. 
	    In particular, if $R$ is the algebra of regular functions over an affine scheme of finite type $X$, we obtain the Lie algebra of all regular maps from $X$ to $\fg$. If $R$ is the ring of meromorphic functions on a Riemann surface with a fixed number of poles we obtain a \emph{current Krichever-Novikov algebra}. These algebras  where introduced by Krichever and Novikov in their study of string theory in Minkowski space \cite{krichever1987algebras}, \cite{Krichever1988} and extensively studied by many authors (cf. \cite{schlichenmaier2014} and  references therein). In particular, the case of the ring of rational functions on the Riemann sphere regular everywhere except a finite number of points is relevant to the study of  tensor module structures for affine Lie algebras (\cite{Kazhdan1991}, \cite{Kazhdan1994}). The corresponding Lie algebras generalizing the untwisted affine Kac-Moody Lie algebras are called {\it{$N$-point algebras}}. 
	    
	   Denote by $\hat{\Gl}$
	   the universal central extension of $\Gl=\fg\ot R$. In \cite{bremner1994universal}, Bremner described the center $C$ of the universal central extension $\hat{\mathcal{G}}$ of $N$-point algebras and  determined its dimension. Also, in the case of the ring of regular functions on an  elliptic curve with two points removed  the universal cocycle $\hat{\mathcal{G}}\times \hat{\mathcal{G}}\to C$ was computed.
	    
	    Following these results,   Cox and Jurisich \cite{Cox2014Realizations}  described the universal central extension of the $3$-point current algebra when $\fg=\mathfrak{sl}(2,\Cx)$ and, furthermore constructed their explicit realizations by differential operators. In \cite{bremner1995four}, Bremner constructed the universal central extension of the $4$-point current algebra. 
	
	    Date, Jimbo, Kashiwara and Miwa considered the universal central extension of $\fg\ot \Cx[t,t^{-1},u]$ with $u^2=(t^2-b^2)(t^2-c^2)$, $b\in\Cx\setminus\{-c,c\}$ in their study of Landau-Lifshitz equation \cite{Date1983} (the so-called  \emph{DJKM algebra}). This is an example of a Krichever-Novikov algebra with genus different from zero. 
	   
	    The universal central extension of the DJKM algebras  was described  in \cite{Cox2011DJKMExtension}. A realization of these algebras in terms  of partial differential operators was constructed in \cite{Cox2014Realizations} for $\fg=\mathfrak{sl}(2,\Cx)$ and in \cite{Cox2014} for a general $\fg$.  A study of the universal central extension of the DJKM algebras led  to the discovery of new families of orthogonal polynomials  in \cite{Cox2013DJKMPolynomials}, which belong to the families of associated ultraspherical polynomials. These non-classical orthogonal polynomials satisfy differential equations of fourth order. Orthogonal polynomials play important role in mathematics and physics, particularly in the description of wave phenomena and vibrating systems, theory of random matrix ensembles and other areas. It is especially interesting to have such a connection  between the universal central extensions of Krichever-Novikov algebras and associated ultraspherical polynomials.
	
	    Building upon the previous research on the universal central extension of the DJKM algebras, we extend the study to the superelliptic case. The equation of a superelliptic curve is $u^m=p(t)$, where $p(t)$ is a polynomial in $t$. The DJKM algebras correspond to the case $m=2$ and $deg \, p(t)=4$. We generalizes the results of \cite{Cox2013DJKMPolynomials} for arbitrary superelliptic curve with $m>2$ and $deg \, p(t)=4$. 
	    Our analysis reveals the existence for each $m>2$ of two new families of orthogonal polynomials which are the solutions of fourth-order differential equations as in the case the DJKM algebras. This is somewhat surprising as we expected the order of differential equations to grow with $m$. 
	   One constructed family is identified as a family of associated ultraspherical polynomials. The other family is similar but has different initial conditions. As a result, it required an independent proof of the orthogonality. 
	    
	   As a byproduct, we show that certain superelliptic integrals are solutions of elliptic PDEs of fourth-order (cf. Corollaries \ref{cor:P_4PDE} and \ref{cor:P_2PDE}). 
	   This is of interest on its own since PDEs we consider are not completely integrable and do not fall into any other class of PDEs with explicit solutions. In the cases when such explicit solutions are known, they are of importance in different areas of mathematics and theoretical physics.
	
	\section{Superelliptic affine Lie algebras.}

	    Let $R$ be a ring of the form $\Cx[t,t^{-1}, u]$, where $u^m\in \Cx[t]$ and $m\geq 2$. Thus $R$ has a basis consisting of $t^i$, $t^iu$, $t^iu^2\dots,t^iu^{m-1}$ for $i\in \Z$. We will assume that $P(t)=\sum^{D}_{i=0}a_it^i$ where $a_D=1$ and $a_0,a_1$ are not both 0 and that $P(t)$ has no multiple roots. The equation $u^m=P(t)$ defines a superelliptic curve. Set $R^i$ for $\Cx[t,t^{-1}]u^i$, then $R=R^0\oplus R^1\oplus \cdots \oplus R^{m-1}$ is a $\Z/m\mathbb{Z}$-grading. 
	
	    Let $\fg$ be a simple finite-dimensional complex Lie algebra. The Lie algebra $\Gl=\fg \ot R$ is an example of \emph{superelliptic loop algebras}. The $\Z/m\mathbb{Z}$-grading induces the structure of a $\Z/m\mathbb{Z}$-graded Lie algebra on $\Gl$ by setting $\Gl^i=\fg\ot R^i$ ($i=0,1,2,\dots, m-1$).
	
	    Let $\hat\Gl=\Gl\oplus C$ be the universal central extension of $\Gl$, where $C$  is the center of $\hat\Gl$. By \cite{kassel1984kahler}, the center $C$ is linearly isomorphic to $\Omega_R^1/dR$, the space of K\"ahler differentials of $R$ modulo the exact differentials. Our goal is to determine a basis for $\Omega_R^1/dR$. 
	
	    Let $F=R\otimes R$ be the left $R$-module with action $f(g\otimes h)=fg\otimes h$ for $f, g, h\in R$. Let $K$ be the submodule generated by the elements $1\otimes fg-f\otimes g-g\otimes f$. Then $\Omega_R^1=F/K$ is the module of K\"ahler differentials. We denote the element $f\otimes g+K$ of $\Omega_R^1$ by $fdg$. We define a map $d:R\rightarrow \Omega_R^1$ by $d(f)=df=1\otimes f+K$ and we denote the coset of $fdg$ modulo $dR$ by $\overline{fdg}$. We have
	        \begin{align}
	        	[x\ot f, y\ot g]=[xy]\otimes fg+(x,y)\overline{fdg},&& [x\ot f,\omega]=0,&& [\omega, \omega']=0
	   		\end{align}
	    where $x,y\in\fg$, $f,g\in R$ and $\omega, \omega' \in\Omega_R^1/dR$; here $(x,y)$ denotes the Killing form on $\fg$. 
	    
	    The elements $t^iu^k\ot t^ju^l$, with $i,j\in\mathbb{Z}$ and $k,l\in \{0,1,\dots, m-1\}$ form a basis of $R\ot R$. The following result is fundamental to the description of the universal central extension for $R=\Cx[t,t^{-1},u|u^m=P(t)]$.
	
	    \begin{theorem}[\cite{AlbinodosSantos2021OnAlgebras}, Theorem 2.4]\label{theorem1}
	       The following elements form a basis of $\Omega_R^1/dR$:   
	       $$\overline{t^{-1}dt}, \, \overline{t^{-1}u^ldt}, \, \dots, \, \overline{t^{-D}u^ldt},$$ where $l\in\{1,2,\dots,m-1\}$ and we omit $\overline{t^{-D}u^ldt}$ if $a_0=0$.
	    \end{theorem}
	    
	    We  have the following relation between the basis elements:
	    
	    \begin{proposition}[\cite{Cox2011DJKMExtension}, Lemma 2.0.2] 
	        If $u^m=P(t)$ and $R=[t,t^{-1},u|u^m=P(t)]$, then in $\Omega_R^1/dR$, one has
	    	\begin{equation}\label{relations_superelliptic}
	            ((m+1)n+im)t^{n+i-1}udt\equiv\sum_{j=0}^{n-1}((m+1)j+mi)a_jt^{i+j-1}udt \mod{dR}.
	        \end{equation}
	    \end{proposition}
	  
	    This allows  to describe the universal central extension of the superelliptic algebra.
 
	    Let us define the sequence of polynomials in $D+3$ parameters $P_{m,n,l}(a_{0},a_{1},\dots,a_{D-1}):=P_{n}$ for $n\geq -D$, $m, l\in\mathbb{Z}_+$ and $a_{0},a_{1},\dots,a_{D-1}\in\Cx$ as follows:
	  	\begin{equation} \label{polinomios}
	    	(Dl+(1+n)m)P_n=\sum_{k=0}^{D-1}\left(-a_k\left(kl+(-D+1+n+k)m\right) \right) P_{-D+n+k},
	    \end{equation}
	        with initial condition
      	\begin{equation} \label{condicoesiniciais}
	    	P_{-D}=t^{-D}u^ldt, \,  P_{-D+1}=t^{-D+1}u^ldt, \, \dots, \, P_{-1}=t^{-1}u^ldt.
	  	\end{equation}
	   	Then $P_n= t^{n}u^ldt$ for all $n\geq 0$.
	      
	   	If $a_0\neq 0$ we define the sequence of polynomials in $D+3$ parameters $Q_{m,n,l}(a_{0},a_{1},\dots,a_{D-1}):=Q_{n}$ for $n\leq -D-1$, $m, l\in\mathbb{Z}_+$ and $a_{0},a_{1},\dots,a_{D-1}\in\Cx$ as follows
	  	\begin{equation} \label{polinomios2}
	   		(a_0(1+n)m)Q_n=\sum_{k=1}^{D}\left(-a_k\left(kl+(1+n+k)m\right) \right) Q_{n+k}
	   	\end{equation}
	  	with initial conditions
	  	\begin{equation} \label{condicoesiniciais2}
	   		Q_{-D}=t^{-D}u^ldt, \, Q_{-D+1}=t^{-D+1}u^ldt, \, \dots, \, Q_{-1}=t^{-1}u^ldt.
	  	\end{equation}
	        
		Then $Q_n=t^{n}u^ldt$ for $n\leq -D-1$.

	  	If $a_0=0$ we define the sequence of polynomials in $D+3$ parameters $R_{m,n,l}(a_{0},a_{1},\dots,a_{D-1}):=R_{n}$ for $n\leq -D-1$, $m, l\in\mathbb{Z}_+$ and $a_{0},a_{1},\dots,a_{D-1}\in\Cx$ as follows:
		\begin{equation} \label{polinomios3}
  			a_1(l+(1+n+1)m)R_n=\sum_{k=2}^{D}\left(-a_k\left(kl+(1+n+k)m\right) \right) R_{n+k}
		\end{equation}
		with initial condition
		\begin{equation} \label{condicoesiniciais3}
			R_{-D+1}=t^{-D+1}u^ldt, \, R_{-D+2}=t^{-D+2}u^ldt, \, \dots, \, R_{-1}=t^{-1}u^ldt.
	 	\end{equation}
		We see that $R_n=t^{n}u^ldt$, $n\leq -D-1$.
	       
	  	Set 
	   	\begin{equation}
	   		\psi_{i,j}=
	       		\begin{cases}
	         		P_{i+j-1} \textrm{ if } i+j\geq -D+1;\\
	          		Q_{i+j-1} \textrm{ if } i+j\leq -D \textrm{ and } a_0\neq0;\\
	             	R_{i+j-1} \textrm{ if } i+j\leq -D\textrm{ and } a_0=0. 
	         	\end{cases}
	   	\end{equation}
	  	Then
	  	$ t^{i}u^ld(t^j) = j \psi_{i,j}$    [\cite{AlbinodosSantos2021OnAlgebras}, Proposition 3.4].

	    We have the following description of the Lie algebra structure on $\hat{\mathcal{G}}$. Set 
	    \begin{equation}
	        \omega_0=\overline{t^{-1}dt} \textrm{ and } \omega_{ij}=\overline{t^{i}u^jdt} \textrm{ for } j\neq 0.
	    \end{equation}
	    
	    \begin{theorem}[\cite{AlbinodosSantos2021OnAlgebras}, Theorem 3.5]\label{corollary}
	            The superelliptic affine Lie algebra $\hat{\mathcal{G}}$ has a $\mathbb{Z}/m\mathbb{Z}$-grading in which
	        \begin{align*}
	            \hat{\mathcal{G}}^0=\fg\ot\Cx[t,t^{-1}]\oplus \Cx\omega_0,&& 
	            \hat{\mathcal{G}}^l=\fg\ot\Cx[t,t^{-1}]u^l\bigoplus_{n=1}^D \Cx\omega_{-n,l}.
	        \end{align*}
	        The subalgebra $\hat{\mathcal{G}}^0$ is an untwisted affine Kac-Moody Lie algebra with commutation relations
	        \begin{equation*}
	            [x\ot t^i,y\ot t^j]=[x,y]\ot t^{i+j}\oplus \delta_{i+j,0}(x,y)j\omega_0.
	        \end{equation*}
	        The commutation relations in $\mathcal{\hat{G}}$ are
	        \begin{equation*}
	            [x\ot t^iu^{l_1},y\ot t^ju^{l_2}]=[x,y]\ot (t^{i+j}u^{l_1+l_2})+(x,y)\left(\frac{jl_1-il_2}{l_1+l_2}\right) \omega_{i+j-1,l_1+l_2}, 
	        \end{equation*}
	        if $l_1+l_2\leq m-1$. When $l_1+l_2>m-1$,
	        \begin{equation*}
	            {}[x\ot t^iu^{l_1},y\ot t^ju^{l_2}]=[x,y]\ot \left(\sum_{k=0}^D a_kt^{i+j+k}u^{l_1+l_2-m}\right)+(x,y)\left(\frac{jl_1-il_2}{l_1+l_2}\right) \omega_{i+j-1,l_1+l_2}.
	        \end{equation*}
	        The subspace $\mathcal{\hat{G}}^l$ is a $\mathcal{\hat{G}}^0$-module with 
	        \begin{align}
	            [x\ot t^iu^n,y\ot t^j]=&[x,y]\ot (t^{i+j}u^{n+1})+(x,y)j\psi_{i,j}.
	        \end{align}
	    \end{theorem}
	  
	      In this paper we will consider the superelliptic affine Lie algebras with the polynomial relation $$u^m=1-2 c t^2+t^4.$$ In this case, letting $k=i+3$, the recursion relation \eqref{relations_superelliptic} becomes 
	    \begin{equation*}
	        (4+m+k m) \overline{t^k \text{udt}}=-(k-3) m \overline{t^{-4+k} \text{udt}}+2 c (2+(k-1) m) \overline{t^{-2+k} \text{udt}}
	    \end{equation*}
	    
	    Consider a family $P_k:=P_k(c)$ of polynomials in $c$ satisfying the recursion relation
	    \begin{equation}
	        (k m+m+4) P_k(c)=2 c ((k-1) m+2) P_{k-2}(c)-(k-3) m P_{k-4}(c)
	    \end{equation}
	    for $k\geq 0$. Set
	    \begin{equation*}
	        P(c,z)=\sum _{k=-4}^{\infty } z^{k+4} P_k(c)=\sum _{k=0}^{\infty } z^{k} P_{k-4}(c).
	    \end{equation*}
	    After a straightforward rearrangement of terms we have
	    \begin{align*}
	        0=&\sum _{k=0}^{\infty } (k m+m+4) z^k P_k(c)-(2 c) \sum _{k=0}^{\infty } (k m-m+2) z^k P_{k-2}(c) \\ 
	        & +\sum _{k=0}^{\infty } (k-3) m z^k P_{k-4}(c)\\= 
	        & ((4 - 3 m)z^{-4} + (2 c (-2 + 3 m)) z^{-2} - 3 m) P(c,z) + m \left(z^{-3}-2cz^{-1}+z \right)\frac{d}{\textrm{d}z} P(c,z) \\
	        &+ \left(2 c (2-3 m) z^{-2}+(3 m-4)z^{-4}\right)P_{-4}(c) + 2 \left(-2 c (m-1) z^{-1}+(m-2)z^{-3}\right)P_{-3}(c) \\
	        &+ (m-4) P_{-2}(c)z^{-2}-4 P_{-1}(c) z^{-1}.
	    \end{align*}
	    Hence, $P(c, z)$ satisfies the differential equation
	    \begin{align}
	        &\frac{d}{\textrm{d}z}P(c,z)+\left(\frac{6 c m z^2-4 c z^2-3 m \left(z^4+1\right)+4}{m z \left(-2 c z^2+z^4+1\right)}\right)P(c,z)=& \nonumber
	        \\&\frac{(4  z^3) P_{-1} + ((-m+4)z^2) P_{-2}+ (2\left(2 c (m-1) z^2-m+2\right)z) P_{-3} +({ \left(-2 c (2-3 m) z^2-3 m+4\right)})P_{-4}}{m \left(-2 c z^3+z^4+z\right)}.
	    \end{align}
	    It has an integrating factor
	    \begin{align*}
	        \mu(c,z)=&\exp{\left\{\int_{\cdot}^z \left(\frac{6 c m w^2-4 c w^2-3 m \left(w^4+1\right)+4}{m w \left(-2 c w^2+w^4+1\right)}\right) \textrm{d}w\right\}} \\
	        =& \frac{1}{z^{3-\frac{4}{m}}(-2 c z^2+z^4+1)^{1/m}}.
	    \end{align*}
	    
	    Now, we consider different cases depending on the initial conditions.\\
	    \subsection{Superelliptic case 1}
	    Assume that $P_{-3}=P_{-2}=P_{-1}=0$ and $P_{-4}=1$. 
	   and denote the generating function in this case by $ P_{-4}(c,z)$. Then 
	    \begin{equation*}
	        P_{-4}(c,z)=\sum _{k=-4}^{\infty }  P_{-4,k}(c)z^{k+4}
	        =\sum _{k=0}^{\infty }P_{-4,k-4}(c) z^{k} 
	    \end{equation*}
	  can be written in terms of a superelliptic integral
	  
	    \begin{align}\label{eqn:P_4IntegralFormula}
	         P_{-4}(c,z)&=z^{3-4/m} (1-2 c z^2+z^4)^{1/m}\times\nonumber\\
	         &\times\lim_{\epsilon\to 0}\left[\int_{\epsilon}^{z} \frac{(4-3 m+2 c(-2+3 m) w^2)}{mw^{4-4/m} (1-2 c w^2+w^4)^{1+1/m}}  \textrm{d}w+\epsilon^{4/m-3}+2c(3-2/m)\phi_m(\epsilon)\right],
	    \end{align}
	    where 
	    \[
	    \phi_m(\epsilon)=
	    \left\{
	\begin{array}{cc}
	    \frac{\epsilon^{4/m-1}}{4/m-1} &  m\neq 4\\
	    \log{\epsilon} &  m=4
	\end{array}
	    \right.
	    \]
	    \begin{remark}
	        Correction term inside of the limit  $\epsilon\to 0$  in the formula \eqref{eqn:P_4IntegralFormula} is necessary to ensure finiteness of the limit.
	    \end{remark}
	    \subsection{Superelliptic case 2}
	    Suppose now that $P_{-3}=P_{-4}=P_{-1}=0$ and $P_{-2}=1$. Then we arrive to the generating function
	    \begin{equation*}
	        P_{-2}(c,z)=\sum _{k=-4}^{\infty }  P_{-2,k}(c)z^{k+4}
	        =\sum _{k=0}^{\infty }P_{-2,k-4}(c) z^{k}, 
	    \end{equation*}
	   which is defined in terms of a superelliptic integral
	    \begin{equation*}
	          P_{-2}(c,z)=z^{3-4/m} (1-2 c z^2+z^4)^{1/m}\lim_{\epsilon\to 0}\left[\int_{\epsilon}^{z} \frac{(4 - m) }{m w^{2-4/m} (1-2 c w^2+w^4)^{1+1/m}}  \textrm{d}w+\epsilon^{4/m-1}\right].
	    \end{equation*}
	    \subsection{Gegenbauer case 3}
	        If we take $P_{-1}(c) = 1$, and $P_{-2}(c) = P_{-3}(c) = P_{-4}(c) = 0$ and set \\ $P_{-1}(c,z)=\sum_{n\geq0}P_{-1,n-4}z^n$, then
	        \begin{equation*}
	            P_{-1}(c,z)=\frac{\left(4 z^{3-\frac{4}{m}} \left(-2 c  z^2+z^4+1\right)^{1/m}\right)}{m} \sum_{n=0}^{\infty} \frac{z^{\frac{4}{m}+2n} Q_n^{\left(1+\frac{1}{m}\right)}(c)}{\frac{4}{m}+2 n},
	        \end{equation*}
	        where $Q_n^{\left(1+\frac{1}{m}\right)}(c)$ is the $n$-th Gegenbauer polynomial.
	
	    \subsection{Gegenbauer case 4}
	    If we take $P_{-3}(c) = 1$, and $P_{-1}(c) = P_{-2}(c) = P_{-4}(c) = 0$ and set $P_{-3}(c,z)=\sum_{n\geq0}P_{-1,n-4}z^n$,
	then
	    \begin{equation*}
	        P_{-3}(c,z)=\frac{\left(z^{3-\frac{4}{m}} \left(-2 c z^2+z^4+1\right)^{1/m}\right)}{2} \sum _{n=0}^{\infty }   z^{\frac{4}{m}+2 n} Q_n^{\left(1+\frac{1}{m}\right)}(c) \left(\frac{m-2}{z^2 (m (n-1)+2)}-\frac{2 c (m-1)}{m n+2}\right).
	    \end{equation*}

	    We will see in the last section that the family of polynomials described in the superelliptic case 1 is an example of associated ultraspherical polynomials.

	    \section{Orthogonal polynomials in superelliptic cases}

	    \subsection{Superelliptic type 1}
	        Let us reindex the polynomials $P_{-4,n}$ as follows:
	    
	        \begin{align*}
	            P_{-4}(c,z)=&z^{3-4/m} (1-2 c z^2+z^4)^{1/m}\times\\
	            &\times\lim_{\epsilon\to 0}\left[\int_{\epsilon}^{z} \frac{(4-3 m+2 c(-2+3 m) w^2)}{mw^{4-4/m} (1-2 c w^2+w^4)^{1+1/m}}  \textrm{d}w+\epsilon^{4/m-3}+2c(3-2/m)\phi_m(\epsilon)\right]=\\
	            &= \sum _{k=0}^{\infty }P_{-4,k-4}(c) z^{k}\\
	            =& 1+\frac{3 m }{m+4}z^4+\frac{6 c m (m+2) }{(m+4) (3 m+4)}z^6 +\frac{3 m  \left(4 c^2 (m+2) (3 m+2)-m (3 m+4)\right)}{(m+4) (3 m+4) (5 m+4)}z^8\\
	            &+\frac{12 c m (3 m+2)  \left(2 c^2 (m+2) (5 m+2)-m (5 m+8)\right)}{(m+4) (3 m+4) (5 m+4) (7 m+4)}z^{10}\\
	            &+3 m\left(\frac{  (16 c^4 (m+2) (3 m+2) (5 m+2) (7 m+2))}{(m+4) (3 m+4) (5 m+4) (7 m+4) (9 m+4)}\right. \\
	            &+\left.\frac{-12 c^2 m (3 m+2) (5 m+2) (7 m+12)+5 m^2 (3 m+4) (7 m+4)}{(m+4) (3 m+4) (5 m+4) (7 m+4) (9 m+4)}\right)z^{12} + O(z^{14}).
	        \end{align*}
	         The first  nonzero polynomials are:
	        \begin{align*}
	            P_{-4,0}(c) = 1,\\
	            P_{-4,4}(c)=&\frac{3 m }{m+4}, \\
	            P_{-4,6}(c)=&\frac{6 c m (m+2) }{(m+4) (3 m+4)}, \\
	            P_{-4,8}(c)=&\frac{3 m  \left(4 c^2 (m+2) (3 m+2)-m (3 m+4)\right)}{(m+4) (3 m+4) (5 m+4)}, \\
	            P_{-4,10}(c)=&\frac{12 c m (3 m+2)  \left(2 c^2 (m+2) (5 m+2)-m (5 m+8)\right)}{(m+4) (3 m+4) (5 m+4) (7 m+4)}, \\
	        \end{align*}
	        and $P_n(c)=P_{-4,n}(c)$ satisfies the  recursion
	        \begin{equation}\label{recursion}
	            0=(2 c ((k-1) m+2)) P_{k+2}(c)-(k m+m+4) P_{k+4}(c)-((k-3) m) P_k(c).
	        \end{equation}
	        
	     Our main result is the following theorem which generalizes  [\cite{Cox2013DJKMPolynomials}, Theorem 3.1.1].
	        
	      \begin{theorem}\label{thm:P_4ODE_System}
	            The polynomials $P_n=P_{-4,n}$ satisfy the following fourth order linear differential equation:
	            \begin{align}\label{eqdiferencial1}
	                0=&16 \left(c^2-1\right)^2 m^2 P^ {(iv)}_n +160 c \left(c^2-1\right) m^2 P'''_n \nonumber\\
	            &+\left(8 m (m ((n-6) n-10)+4 (n-6))\right.\nonumber\\
	            &\left.-8 c^2 \left(m^2 (n-10) (n+4)+4 m (n-4)+8\right)\right) P''_n \nonumber\\
	            &-24 c \left(m^2 (n-6) n+4 m (n-4)+8\right)  P'_n \nonumber\\
	            &+(n-4) n (m (n-6)+4) (m (n-2)+4) P_n. 
	            \end{align}
	        \end{theorem}    
	        
	   \begin{proof}        
	  The proof generalizes the proof of Theorem 3.1.1 in \cite{Cox2013DJKMPolynomials} in  the case $m=2$.   We start off with the generating function
	        \begin{align*}
	            P_{-4}(c,z)&=z^{3-4/m} (1-2 c z^2+z^4)^{1/m}\\
	            &\times\lim_{\epsilon\to 0}\left[\int_{\epsilon}^{z}\frac{4 c \left(w^2 \left(\left(\frac{3 m}{2}-1\right) w^{\frac{4}{m}-2}\right)\right)-(3 m-4) w^{\frac{4}{m}-2}}{m w^2 \left(-2 c w^2+w^4+1\right)^{\frac{1}{m}+1}}  \textrm{d}w+\epsilon^{4/m-3}+2c(3-2/m)\phi_m(\epsilon)\right] \\
	            &=z^{3-4/m} (1-2 c z^2+z^4)^{1/m} \left(-\sum _{n=0}^{\infty } \frac{z^{\frac{4}{m}+2 n-3} \left((3 m-4) Q_n^{\left(1+\frac{1}{m}\right)}(c)\right)}{m (2 n-3)+4}+\right.\\
	            &+\left. \sum _{n=0}^{\infty } \frac{z^{\frac{4}{m}+2 n-1} \left(c (6 m-4) Q_n^{\left(1+\frac{1}{m}\right)}(c)\right)}{m (2 n-1)+4}\right),
	        \end{align*}
	        where $Q_n^\lambda(c)$ is the $n$th-Gegenbauer polynomial. The polynomials $Q_n^\lambda(c)$ satisfy the second order linear differential equation
	        \begin{equation*}
	            \left(1-c^2\right) y''-c (2 \lambda +1) y'+n y (2 \lambda +n)=0,
	        \end{equation*}
	        where the derivatives are with respect to $c$. Thus for $\lambda = 1+\frac{1}{m}$ we get
	        \begin{equation}
	            \left(1-c^2\right) (Q_n^{\left(1+\frac{1}{m}\right)})''(c)-c \left(\frac{2}{m}+3\right) (Q_n^{\left(1+\frac{1}{m}\right)})'(c)+n \left(\frac{2}{m}+n+2\right) Q_n^{\left(1+\frac{1}{m}\right)}(c)=0.
	        \end{equation}
	        Rewriting the expansion formula for $P_{-4}(c,z)$ we get
	        \begin{align}\label{eq3.3}
	            z^{-3+4/m} (1-2 c z^2+z^4)^{-1/m}P_{-4}(c,z)&= \sum _{n=0}^{\infty } \frac{z^{\frac{4}{m}+2 n-1} \left(c (6 m-4) Q_n^{\left(1+\frac{1}{m}\right)}(c)\right)}{m (2 n-1)+4} \nonumber\\
	            &- \sum _{n=0}^{\infty } \frac{z^{\frac{4}{m}+2 n-3} \left((3 m-4) Q_n^{\left(1+\frac{1}{m}\right)}(c)\right)}{m (2 n-3)+4}.
	        \end{align}
	        Now we apply the differential operator $L:=(1-c^2)\frac{d^2}{dc^2}-(c(3+\frac{2}{m}))\frac{d}{dc}$ to the right hand side and  use the identity  (formula 4.7.27 in \cite{Szego1939})
	        \begin{equation}
	            \left(1-c^2\right) \diff{}{c}Q_n^{(\lambda )}(c)=c (2 \lambda +n) Q_n^{(\lambda )}(c)-(n+1) Q_{n+1}^{(\lambda )}(c).
	        \end{equation}
	        Then we get
	        \begin{align*}
	            L \left(c(-\frac{4}{m}+6) Q_n^{\left(1+\frac{1}{m}\right)}(c)\right)=&\left( (1-c^2)\frac{d^2}{dc^2}-(c(3+2/m))\frac{d}{dc} \right)\left(c(-4/m+6) Q_n^{\left(1+\frac{1}{m}\right)}(c)\right)\\
	            =&-\frac{2 (3 m-2)}{m} \left(2 m (n+1) Q_{n+1}^{\left(1+\frac{1}{m}\right)}(c)\right.
	            \\+&\left. c (n-1) (m n+m+2) Q_n^{\left(1+\frac{1}{m}\right)}(c)\right).
	        \end{align*}
	       
	        From here we deduce
	        \begin{align*}
	            &L \left(\sum _{n=0}^{\infty } \frac{z^{\frac{4}{m}+2 n-1} \left(c (6 m-4) Q_n^{\left(1+\frac{1}{m}\right)}(c)\right)}{m (2 n-1)+4}\right)\\
	            =&\sum _{n=0}^{\infty } -\frac{z^{\frac{4}{m}+2 n-1} \left((2 (3 m-2)) \left(2 m (n+1) c_{n+1}^{\left(1+\frac{1}{m}\right)}(c)+c (n-1) (m n+m+2) Q_n^{\left(1+\frac{1}{m}\right)}(c)\right)\right)}{m (m (2 n-1)+4)}\\
	            =&\sum _{n=0}^{\infty } -\frac{z^{\frac{4}{m}+2 n-1} \left((2 c (3 m-2) (n-1) (m n+m+2)) Q_n^{\left(1+\frac{1}{m}\right)}(c)\right)}{m (m (2 n-1)+4)}\\
	             &+\sum _{n=0}^{\infty } -\frac{z^{\frac{4}{m}+2 n-1} \left((4 (3 m-2) (n+1)) c_{n+1}^{\left(1+\frac{1}{m}\right)}(c)\right)}{m (2 n-1)+4}.
	        \end{align*}
	        Since
	        \begin{align*}
	            &\sum _{n=0}^{\infty } -\frac{z^{\frac{4}{m}+2 n-1} \left((2 c (3 m-2) (n-1) (m n+m+2)) Q_n^{\left(1+\frac{1}{m}\right)}(c)\right)}{m (m (2 n-1)+4)}\\
	            &=\frac{6 c (3 m-2) (m+4) \int_{0}^z w^{\frac{4}{m}-2} \left(-2 c w^2+w^4+1\right)^{-\frac{1}{m}-1} \, dw}{4 m^2}\\
	            &-\frac{c (3 m-2) z^{\frac{4}{m}-1} \left(-2 c z^2+z^4+1\right)^{-\frac{1}{m}-1}}{2 m}-\frac{\left(2 c (3 m-2) z^{4/m}\right) \diff{}{z}\left(-2 c z^2+z^4+1\right)^{-\frac{1}{m}-1}}{4 m},
	        \end{align*}
	        and
	        \begin{align*}
	            &\sum _{n=0}^{\infty } -\frac{z^{\frac{4}{m}+2 n-1} \left((4 (3 m-2) (n+1)) Q_{n+1}^{\left(1+\frac{1}{m}\right)}(c)\right)}{m (2 n-1)+4}\\
	            &=-\frac{(2 (3 m-4) (3 m-2)) \int_{0}^z w^{\frac{4}{m}-4} \left(\left(-2 c w^2+w^4+1\right)^{-\frac{1}{m}-1}-1\right) \, dw}{m^2}\\
	            &-\frac{(4 (3 m-2)) \left(z^{\frac{4}{m}-3} \left(\left(-2 c z^2+z^4+1\right)^{-\frac{1}{m}-1}-1\right)\right)}{2 m},
	        \end{align*}
	        we get
	        \begin{align*}
	            &L \left(\sum _{n=0}^{\infty } \frac{z^{\frac{4}{m}+2 n-1} \left(c (6 m-4) Q_n^{\left(1+\frac{1}{m}\right)}(c)\right)}{m (2 n-1)+4}\right)\\
	            &=-\frac{c (3 m-2) z^{\frac{4}{m}-1} \left(-2 c z^2+z^4+1\right)^{-\frac{1}{m}-1}}{2 m}\\
	            &-\frac{\left((2 c) (3 m-2) z^{4/m}\right) \diff{}{z}\left(-2 c z^2+z^4+1\right)^{-\frac{1}{m}-1}}{4 m}\\
	            &+\frac{((2 c) (3 m-2) (3 (m+4))) \int_{0}^z w^{\frac{4}{m}-2} \left(-2 c w^2+w^4+1\right)^{-\frac{1}{m}-1} \, dw}{4 m^2}\\
	            &-\frac{(2 (3 m-4) (3 m-2)) \int_{0}^z w^{\frac{4}{m}-4} \left(-2 c w^2+w^4+1\right)^{-\frac{1}{m}-1} \, dw}{m^2}\\
	            &+\frac{(4-6 m) z^{\frac{4}{m}-3} \left(-2 c z^2+z^4+1\right)^{-\frac{1}{m}-1}}{m}.
	        \end{align*}
	        In addition we have
	        \begin{align*}
	            &L \left(\sum _{n=0}^{\infty } \frac{z^{\frac{4}{m}+2 n-3} \left((3 m-4) Q_n^{\left(1+\frac{1}{m}\right)}(c)\right)}{m (2 n-3)+4}\right)\\
	            &=-\sum _{n=0}^{\infty } \frac{z^{\frac{4}{m}+2 n-3} \left((3 m-4) \left(n \left(\frac{2}{m}+n+2\right)\right) Q_n^{\left(1+\frac{1}{m}\right)}(c)\right)}{m (2 n-3)+4}\\
	            &-\sum _{n=0}^{\infty } \left(\frac{7 \left(9 m^2-24 m+16\right)}{4 m (2 m n-3 m+4)}+\frac{(3 m-4) n}{2 m}+\frac{7 (3 m-4)}{4 m}\right) z^{\frac{4}{m}+2 n-3}\\
	            &-\frac{7 (3 m-4) z^{\frac{4}{m}-3} \left(-2 c z^2+z^4+1\right)^{-\frac{1}{m}-1}}{4 m}-\frac{(3 m-4) z^{\frac{4}{m}-2} \diff{}{z}\left(-2 c z^2+z^4+1\right)^{-\frac{1}{m}-1}}{4 m}\\
	            &-\frac{\left(7 \left(9 m^2-24 m+16\right)\right) \int_{0}^z w^{\frac{4}{m}-4} \left(-2 c w^2+w^4+1\right)^{-\frac{1}{m}-1} \, dw}{4 m^2}.
	        \end{align*}
	     Applying the  operator $L$ to the left hand side of \eqref{eq3.3}, we get
	        \begin{align*}
	            &L \left(z^{\frac{4}{m}-3} \left(-2 c z^2+z^4+1\right)^{-1/m} P_{-4}(c,z)\right)\\
	            &=\frac{z^{\frac{4}{m}-1} \left(-2 c z^2+z^4+1\right)^{-\frac{1}{m}-2} \left(4 c^2 (2 m+1) z^2-2 c (3 m+2) \left(z^4+1\right)+4 (m+1) z^2\right) P_{-4}(c,z)}{m^2}\\
	            &+\frac{z^{\frac{4}{m}-3} \left(-2 c z^2+z^4+1\right)^{-\frac{1}{m}-1} \left(6 c^2 m z^2-c (3 m+2) \left(z^4+1\right)+4 z^2\right) P'_{-4}(c,z)}{m}\\
	            &-\left(c^2-1\right) z^{\frac{4}{m}-3} \left(-2 c z^2+z^4+1\right)^{-1/m} P''_{n-4}(c,z).
	        \end{align*}
	        Hence, we obtain 
	        \begin{align*}
	            &-\frac{3 (m+4) z^{\frac{4}{m}-3} \left(-2 c z^2+z^4+1\right)^{-\frac{1}{m}-1}}{4 m}-\frac{c (3 m-2) z^{\frac{4}{m}-1} \left(-2 c z^2+z^4+1\right)^{-\frac{1}{m}-1}}{2 m}\\
	            &+\frac{(m+1) (3 m-4) \left(c-z^2\right) z^{\frac{4}{m}-1} \left(-2 c z^2+z^4+1\right)^{-\frac{1}{m}-2}}{m^2}\\
	            &-\frac{2 c (m+1) (3 m-2) \left(c-z^2\right) z^{\frac{m+4}{m}} \left(-2 c z^2+z^4+1\right)^{-\frac{1}{m}-2}}{m^2}\\
	            &+\frac{3 (m+4)}{4 m}\left(z^{\frac{4}{m}-3}\left(-2 c z^2+z^4+1\right)^{-1/m}P_{-4}(c,z)\right)\\
	            &=\frac{z^{\frac{4}{m}-1} \left(-2 c z^2+z^4+1\right)^{-\frac{1}{m}-2} \left(4 c^2 (2 m+1) z^2-2 c (3 m+2) \left(z^4+1\right)+4 (m+1) z^2\right) P_{-4}(c,z)}{m^2}\\
	            &+\frac{z^{\frac{4}{m}-3} \left(-2 c z^2+z^4+1\right)^{-\frac{1}{m}-1} \left(6 c^2 m z^2-c (3 m+2) \left(z^4+1\right)+4 z^2\right) P'_{-4}(c,z)}{m}\\
	            &-\left(c^2-1\right) z^{\frac{4}{m}-3} \left(-2 c z^2+z^4+1\right)^{-1/m} P''_{n-4}(c,z),
	        \end{align*}
	        which gives us
	        \begin{align*}
	            &\frac{1}{4} \left(-z^4 \left(4 c^2 (m+2) (3 m-2)+(3 m+4) (5 m-4)\right)+2 c \left(9 m^2+6 m-8\right) z^6\right. \\
	            &\left.+4 c (3 m (m+2)-4) z^2-3 m (m+4)\right)\\
	            &=-\left(m \left(4 c^2 z^4-6 c \left(z^6+z^2\right)+3 z^8+2 z^4+3\right)+4 z^2 \left(-\left(c^2+1\right) z^2+c z^4+c\right)\right.\\
	            &\left.+\frac{3}{4} m^2 \left(-2 c z^2+z^4+1\right)^2\right) P_{-4}(c,z)\\
	            &+m \left(-2 c z^2+z^4+1\right) \left(6 c^2 m z^2-c (3 m+2) \left(z^4+1\right)+4 z^2\right) P'_{-4}(c,z)\\
	            &-\left(c^2-1\right) m^2 \left(-2 c z^2+z^4+1\right)^2 P''_{-4}(c,z).
	        \end{align*}
	        Expanding this out  and writing $P_{-4,k}(c)$ as $P_k$ we will have the following equality:
	        \begin{align}
	            0&=\left(c^2 (-(m+2)) (3 m-2)-\frac{1}{2} m (3 m+4)+4\right) P_{n-4} +(c (3 m (m+2)-4)) P_{n-2} \nonumber\\
	            &-\frac{1}{4} (3 m (m+4)) P_n +\frac{1}{4} (-3 m (m+4)) P_{n-8} +(c (3 m (m+2)-4)) P_{n-6} \nonumber\\
	            &-(c m (3 m+2)) P'_n +\left(4 \left(c^2 m (3 m+1)+m\right)\right) P'_{n-2} \nonumber\\
	            &-\left(6 c m \left(2 c^2 m+m+2\right)\right) P'_{n-4} +\left(4 \left(c^2 m (3 m+1)+m\right)\right) P'_{n-6} \nonumber\\
	            &+(-c m (3 m+2)) P'_{n-8} -\left(\left(c^2-1\right) m^2\right) P''_n +\left(4 c \left(c^2-1\right) m^2\right) P''_{n-2} \nonumber\\
	            &-\left(\left(c^2-1\right) m^2\right) P''_{n-8} +\left(4 c \left(c^2-1\right) m^2\right) P''_{n-6} \nonumber\\
	            &+\left(2 \left(-2 c^4+c^2+1\right) m^2\right) P''_{n-4} . \label{A.eq.3.5.COX}
	        \end{align}
	        We now differentiate twice with respect to $c$ the recursion \ref{recursion}:             
	        \begin{align}
	            0=&2 ((k-1) m+2) P_{k+2} +2 c ((k-1) m+2) P'_{k+2} \nonumber\\
	            &-(k-3) m P'_k -(k m+m+4) P'_{k+4} \label{A.eq.3.7.COX}\\
	            0=&(4 (k-1) m+8) P'_{k+2} +2 c ((k-1) m+2) P''_{k+2} \nonumber\\
	            &-(k m+m+4) P''_{k+4} -(k-3) m P''_k .\label{A.eq.3.8.COX}
	        \end{align}
	        Substituting $k=n-8$ in the last equation we obtain
	        \begin{align}
	            0=&4 (m (n-9)+2) P'_{n-6} -m (n-11) P''_{n-8} \nonumber\\
	            &+2 c (m (n-9)+2) P''_{n-6} -(m (n-7)+4) P''_{n-4} .
	        \end{align}
	        Multiplying \eqref{A.eq.3.5.COX} by $-11+n$  and adding it to the above equation multiplied by $(1-c^2) m$  gives us
	        \begin{align}
	            0=&4 c \left(c^2-1\right) m^2 (n-11) P''_{n-2} -\left(c^2-1\right) m^2 (n-11) P''_n \nonumber\\
	            &+4 m \left(c^2 (2 m (n-12)+n-13)+(m+1) (n-9)\right) P'_{n-6} \nonumber\\
	            &-6 c m (n-11) \left(2 c^2 m+m+2\right) P'_{n-4} \nonumber\\
	            &+2 c \left(c^2-1\right) m (m (n-13)-2) P''_{n-6} \nonumber\\
	            &-\left(c^2-1\right) m \left(m \left(4 c^2 (n-11)+n-15\right)-4\right) P''_{n-4} \nonumber\\
	            &+(n-11) \left(c^2 (-(m+2)) (3 m-2)-\frac{1}{2} m (3 m+4)+4\right) P_{n-4} \nonumber\\
	            &+4 (n-11) \left(c^2 m (3 m+1)+m\right) P'_{n-2} -\frac{3}{4} (m+4) m (n-11) P_{n-8} \nonumber\\
	            &-\frac{3}{4} (m+4) m (n-11) P_n -c (3 m+2) m (n-11) P'_{n-8} \nonumber\\
	            &-c (3 m+2) m (n-11)  P'_n +c (3 m (m+2)-4) (n-11) P_{n-6} \nonumber\\
	            &+c (3 m (m+2)-4) (n-11) P_{n-2} .
	        \end{align}
	        Now substitute $k=n-8$ in \eqref{A.eq.3.7.COX} to get
	        \begin{align}
	            0=&2 (m (n-9)+2) P_{n-6} -m (n-11) P'_{n-8} +2 c (m (n-9)+2) P'_{n-6} \nonumber\\
	            &-(m (n-7)+4) P'_{n-4} .
	        \end{align}
	        Multiplying this equation by $-c (2+3 m)$ and adding it to the previous equation we obtain 
	        \begin{align*}
	            0=&4 c \left(c^2-1\right) m^2 (n-11) P''_{n-2} -\left(c^2-1\right) m^2 (n-11) P''_n \\
	            &+2 c \left(c^2-1\right) m (m (n-13)-2) P''_{n-6} \\
	            &+\left(1-c^2\right) m \left(m \left(4 c^2 (n-11)+n-15\right)-4\right) P''_{n-4} \\
	            &+(n-11) \left(c^2 (-(m+2)) (3 m-2)-\frac{1}{2} m (3 m+4)+4\right) P_{n-4} \\
	            &+\left(2 c^2 (m (m (n-21)-14)-4)+4 m (m+1) (n-9)\right) P'_{n-6} \\
	            &+\left(c (3 m+2) (m (n-7)+4)-6 c m (n-11) \left(2 c^2 m+m+2\right)\right) P'_{n-4} \\
	            &+4 (n-11) \left(c^2 m (3 m+1)+m\right) P'_{n-2} \\
	            &-\frac{3}{4} (m+4) m (n-11) P_{n-8} -\frac{3}{4} (m+4) m (n-11) P_n \\
	            &-c (3 m+2) m (n-11)  P'_n -c (m (3 m (n-7)-2 n+42)+4 (n-9)) P_{n-6} \\
	            &+c (3 m (m+2)-4) (n-11) P_{n-2}.
	        \end{align*}
	        Finally, substitute $k=n-8$ in \eqref{recursion}, multiply it by $-3 (4+m)$ and add  to the previous equation multiplied by $4$:
	        \begin{align*}
	            0=&-2 c (m (m (9 n-69)+8 n-18)+8 (n-6)) P_{n-6} \\
	            &+\left(-4 c^2 (m+2) (3 m-2) (n-11)+16 (m+n-8)+m (4 n-3 m (n-15))\right) P_{n-4} \\
	            &+4 c (3 m (m+2)-4) (n-11) P_{n-2} -3 m (m+4) (n-11) P_n \\
	            &+\left(8 c^2 (m (m (n-21)-14)-4)+16 m (m+1) (n-9)\right) P'_{n-6} \\
	            &+4 \left(c (3 m+2) (m (n-7)+4)-6 c m (n-11) \left(2 c^2 m+m+2\right)\right) P'_{n-4} \\
	            &+16 (n-11) \left(c^2 m (3 m+1)+m\right) P'_{n-2} -4 c m (3 m+2) (n-11)  P'_n \\
	            &+8 c \left(c^2-1\right) m (m (n-13)-2) P''_{n-6} \\
	            &+4 \left(1-c^2\right) m \left(m \left(4 c^2 (n-11)+n-15\right)-4\right) P''_{n-4} \\
	            &+16 c \left(c^2-1\right) m^2 (n-11) P''_{n-2} -4 \left(c^2-1\right) m^2 (n-11) P''_n .
	        \end{align*}
	        We obtained the equation without terms with index $n-8$. 
	        
	  The next step is to eliminate  $P_{n-6}$. After setting $k=n-6$ in \eqref{A.eq.3.8.COX}, \eqref{A.eq.3.7.COX} and in the recursion we obtain the following equations:
	        \begin{align}
	            0=&4 (m (n-7)+2) P'_{n-4} -m (n-9) P''_{n-6} \nonumber\\
	            &+2 c (m (n-7)+2) P''_{n-4} -(m (n-5)+4) P''_{n-2},
	        \end{align}
	        
	        \begin{align}
	            0=&2 (m (n-7)+2) P_{n-4} -m (n-9) P'_{n-6} \nonumber\\
	            &+2 c (m (n-7)+2) P'_{n-4} -(m (n-5)+4) P'_{n-2},
	        \end{align}
	         \begin{align*}
	            0=-m (n-9) P_{n-6} +2 c (m (n-7)+2) P_{n-4} -(m (n-5)+4) P_{n-2}.
	        \end{align*}
	        
	        Proceeding as in the previous case and using these three equations we can eliminate 
	        all terms with index $n-6$: 
	
	        \begin{align*}
	            0=&\left(-64 c^3 (m+1) (2 m+1) (m (n-4)+2)-4 c m (m (m (n (3 n-104)+629)\right.\\
	            &\left.+2 (n-54) n+714)-24 n+184)\right) P'_{n-4} \\
	            &-4 \left(c^2-1\right) m^3 (n-11) (n-9) P''_n \\
	            &+\left(m (n-9) \left(m^2 (29 n-179)+36 m (n-4)+16 (n-4)\right)\right.\\
	            &\left.-16 c^2 (m+1) (2 m+1) (n-4) (m (n-6)+4)\right) P_{n-4} \\
	            &+m \left((m (n-5)+4) \left(8 c^2 \left(-m (n-21)+\frac{4}{m}+14\right)-16 (m+1) (n-9)\right)\right.\\
	            &\left.+16 (n-11) (n-9) \left(c^2 m (3 m+1)+m\right)\right) P'_{n-2} \\
	            &-4 \left(c^2-1\right) m \left(16 c^2 (m+1) (2 m+1)+m (n-9) (m (n-15)-4)\right) P''_{n-4} \\
	            &+8 c \left(c^2-1\right) m (m (m ((n-22) n+133)-2 n+42)+8) P''_{n-2} \\
	            &-3 (m+4) m^2 (n-11) (n-9) P_n -4 c (3 m+2) m^2 (n-11) (n-9)  P'_n \\
	            &+(4 c m (3 m (m+2)-4) (n-11) (n-9)\\
	            &+2 c (m (n-5)+4) (m (m (9 n-69)+8 n-18)+8 (n-6))) P_{n-2}.
	        \end{align*}
	     
	     Next step is to eliminate terms with index $n-4$. Substituting $k=n-6$ in \eqref{A.eq.3.8.COX}, \eqref{A.eq.3.7.COX} and in the recursion we obtain the following equations:
	        \begin{align*}
	            0=&4 (m (n-5)+2) P'_{n-2} -m (n-7) P''_{n-4} \\
	            &+2 c (m (n-5)+2) P''_{n-2} -(m (n-3)+4) P''_n,
	        \end{align*}
	      
	        \begin{align*}
	            0=&2 (m (n-5)+2) P_{n-2} -m (n-7) P'_{n-4} \\&+2 c (m (n-5)+2) P'_{n-2} -(m (n-3)+4)  P'_n,
	        \end{align*}
	       
	        \begin{align*}
	            0=-m (n-7) P_{n-4} +2 c (m (n-5)+2) P_{n-2} -(m (n-3)+4) P_n.
	                    \end{align*}
	       
	       Using these equations we can eliminate terms with index $n-4$:
	       
	        \begin{align}
	            0=&2 c (n-2) (m (n-4)+4) \left(c^2 (m (n-5)+2)-m (n-8)\right) P_{n-2} \nonumber\\
	            &-(n-4) (m (n-6)+4) \left(c^2 (m (n-3)+4)-m (n-9)\right) P_n \nonumber\\
	            &-4 c \left(c^2 (m (n-4)+2) (m (n-3)+4)+m (m (6-(n-7) n)-6 n+46)\right)  P'_n \nonumber\\
	            &+8 c \left(c^2-1\right) m \left(c^2 (m (n-5)+2)-m (n-8)\right) P''_{n-2} \nonumber\\
	            &-4 \left(c^2-1\right) m \left(c^2 (m (n-3)+4)-m (n-9)\right) P''_n \nonumber\\
	            &+4 \left(2 c^4 (m (n-5)+2) (m (n-2)+2)+c^2 m (m ((28-3 n) n-47)-8 n+52)\right.\nonumber\\
	            &\left.+m^2 (n-9) (n-5)\right) P'_{n-2}.\label{A.eq.3.14.Cox}
	        \end{align}
	        
	      Finally, we will eliminate all terms with index $n-2$. Substituting $k=n-2$ in \eqref{A.eq.3.8.COX}, \eqref{A.eq.3.7.COX} and in the recursion we obtain:

	        \begin{align*}
	            0=&4 (m (n-3)+2)  P'_n -m (n-5) P''_{n-2} \\
	            &+2 c (m (n-3)+2) P''_n +(m (-n)+m-4) P''_{n+2},
	        \end{align*}
	              
	        \begin{align*}
	            0=&2 (m (n-3)+2) P_n -m (n-5) P'_{n-2} +2 c (m (n-3)+2)  P'_n \\
	            &+(m (-n)+m-4) P'_{n+2},
	        \end{align*}
	       
	        \begin{align*}
	            0=-m (n-5) P_{n-2} +2 c (m (n-3)+2) P_n +(m (-n)+m-4) P_{n+2}.
	        \end{align*}

	These equations substitution $n\rightarrow n-2$ lead to the following equation:
	        \begin{align}
	            0=&2 c (n-4) (m (n-6)+4) (-m (n-3)-4) \left(c^2 (m (n-7)+2)-m (n-10)\right) P_n \nonumber \\
	            &+8 c \left(c^2-1\right) m (-m (n-3)-4) \left(c^2 (m (n-7)+2)-m (n-10)\right) P''_n \nonumber \\
	            &+4 c \left(4 c^4 (m (n-7)+2) (m (n-5)+2) (m (n-2)+2)\right.\nonumber\\
	            &\left.+c^2 m \left(m \left(-m (n-5) (7 (n-11) n+136)-34 n^2+404 n-1026\right)-40 n+296\right)\right.\nonumber\\
	            &\left.+m^2 (m (n (n (3 n-56)+303)-454)+2 n (5 n-78)+554)\right) P'_{n-2} \nonumber \\
	            &+(n-2) (m (n-4)+4) \left(4 c^4 (m (n-7)+2) (m (n-5)+2)\right.\nonumber\\
	            &\left.+c^2 m (-m (n-5) (5 n-47)-12 (n-9))+m^2 (n-11) (n-7)\right) P_{n-2} \nonumber \\
	            &+4 \left(c^2-1\right) m \left(4 c^4 (m (n-7)+2) (m (n-5)+2)\right.\nonumber\\
	            &\left.+c^2 m (-m (n-5) (5 n-47)-12 (n-9))+m^2 (n-11) (n-7)\right) P''_{n-2} \nonumber \\
	            &+m (-m (n-3)-4) \left(\frac{8 c^4 (m (n-7)+2) (m (n-4)+2)}{m}\right.\nonumber\\
	            &\left.-4 c^2 (m (n (3 n-40)+115)+8 n-68)+4 m (n-11) (n-7)\right)  P'_n.\label{A.eq.3.15.Cox}
	        \end{align}
	        Combining this with  \eqref{A.eq.3.14.Cox} we obtain
	        \begin{align}
	            0=&2 c (n-2) (m (n-4)+4) \left(c^2 (m (n-5)+2)-m (n-8)\right) P_{n-2} \nonumber\\
	            &-(n-4) (m (n-6)+4) \left(c^2 (m (n-3)+4)-m (n-9)\right) P_n \nonumber\\
	            &-4 c \left(c^2 (m (n-4)+2) (m (n-3)+4)+m (m (6-(n-7) n)-6 n+46)\right)  P'_n \nonumber\\
	            &+8 c \left(c^2-1\right) m \left(c^2 (m (n-5)+2)-m (n-8)\right) P''_{n-2} \nonumber\\
	            &-4 \left(c^2-1\right) m \left(c^2 (m (n-3)+4)-m (n-9)\right) P''_n \nonumber\\
	            &+4 \left(2 c^4 (m (n-5)+2) (m (n-2)+2)+c^2 m (m ((28-3 n) n-47)-8 n+52)\right.\nonumber\\
	            &\left.+m^2 (n-9) (n-5)\right) P'_{n-2}. \label{A.eq.3.16.Cox}
	        \end{align}
	       
	       Now we use  \eqref{A.eq.3.16.Cox} and \eqref{A.eq.3.15.Cox} to simplify the coefficient by $P''_{n-2}$. Combining them 
	        Multiplying \eqref{A.eq.3.16.Cox} by $-2 c (mn-7m+2)$ and adding it to \eqref{A.eq.3.15.Cox} gives:
	        \begin{align*}
	            0=&-4 \left(c^2-1\right) m^2 (n-11) \left(c^2 (m (n-1)+4)-m (n-7)\right) P''_{n-2} \\
	            &+8 c \left(c^2-1\right) (3 m+2) m^2 (n-11) P''_n \\
	            &+4 m (n-11) \left(c^2 \left(m^2 ((n-10) n+39)+4 m (n-1)+8\right)-m (n-7) (m (n-3)+4)\right)  P'_n \\
	            &-4 c m (n-11) \left(c^2 (m (n-2)+2) (m (n-1)+4)+m (m (16-(n-3) n)-6 n+34)\right) P'_{n-2} \\
	            &-m (n-11) (n-2) (m (n-4)+4) \left(c^2 (m (n-1)+4)-m (n-7)\right) P_{n-2} \\
	            &+2 c (3 m+2) m (n-11) (n-4) (m (n-6)+4) P_n.
	        \end{align*}
	        
	      From now on we assume   
	        that $n> 11$. Then  the equation simplifies to
	        \begin{align}
	            0=&\left(4 m (n-7) (m (n-3)+4)-4 c^2 \left(m^2 ((n-10) n+39)+4 m (n-1)+8\right)\right)  P'_n \nonumber\\
	            &+(n-2) (m (n-4)+4) \left(c^2 (m (n-1)+4)-m (n-7)\right) P_{n-2} \nonumber\\
	            &+4 c \left(c^2 (m (n-2)+2) (m (n-1)+4)+m (m (16-(n-3) n)-6 n+34)\right) P'_{n-2} \nonumber\\
	            &+4 \left(c^2-1\right) m \left(c^2 (m (n-1)+4)-m (n-7)\right) P''_{n-2} \nonumber\\
	            &-8 c \left(c^2-1\right) m (3 m+2) P''_n -2 c (3 m+2) (n-4) (m (n-6)+4) P_n. \label{A.eq.3.17.Cox}
	        \end{align}
	   Multiplying it by $-2 c (m n-5m+2)$ and adding it to  \eqref{A.eq.3.14.Cox} multiplied by $4+mn-m $, we get:
	       
	        \begin{align}
	            0=&\left(4 c^2 \left(m^2 ((n-6) n+23)+4 m (n+1)+8\right)-4 m (n-5) (m (n-1)+4)\right) P'_{n-2} \nonumber\\
	            &+(n-4) (m (n-6)+4) \left(c^2 m (n-7)-m n+m-4\right) P_n \nonumber\\
	            &-4 c \left(m \left(c^2 (n-7) (m (n-6)+2)-m ((n-13) n+24)-2 n+38\right)+8\right)  P'_n \nonumber\\
	            &+8 c \left(c^2-1\right) m (3 m+2) P''_{n-2} \nonumber\\
	            &+4 \left(c^2-1\right) m \left(c^2 m (n-7)-m n+m-4\right) P''_n \nonumber\\
	            &+2 c (3 m+2) (n-2) (m (n-4)+4) P_{n-2}. \label{A.eq.3.18.Cox}
	        \end{align}
	        
	       Multiplying this equation by $c (-m n+m-4)$ and adding it to  \eqref{A.eq.3.17.Cox} multiplied 
	       by $6 m+4$, we get:
	     
	        \begin{align}
	            0=&c (n-4) (m (n-6)+4) \left(c^2 (m (n-1)+4)-m (n+5)-8\right) P_n \nonumber\\
	            &+4 c \left(m \left(c^2 (n-5) (m (n-1)+4)+m (13-(n-6) n)-4 n+44\right)+8\right) P'_{n-2} \nonumber\\
	            &+8 \left(c^2-1\right) m (3 m+2) P''_{n-2} \nonumber\\
	            &+4 c \left(c^2-1\right) m \left(c^2 (m (n-1)+4)-m (n+5)-8\right) P''_n \nonumber\\
	            &+\left(-4 c^4 (m (n-6)+2) (m (n-1)+4)+4 c^2 (m (n-6)+2) (m (n+5)+8)\right.\nonumber\\
	            &\left.-8 (3 m+2) (m (n-3)+4)\right)  P'_n .\nonumber\\
	            &+2 (3 m+2) (n-2) (m (n-4)+4) P_{n-2}. \label{A.eq.3.19.Cox}
	        \end{align}
	       
	Multiplying it by $-c$ and adding  to \eqref{A.eq.3.18.Cox} we obtain the following equation:
	     
	        \begin{align}
	            0=&4 \left(c^2-1\right) m P''_n +(n-4) (m (n-6)+4) P_n \nonumber \\
	            &+4 m (n-5) P'_{n-2} -4 c (m (n-6)+2)  P'_n. \label{A.eq.3.20.Cox}
	        \end{align}
	        
	        Differentiate the last equation with respect to $c$ to get
	        \begin{align}
	            0=&4 \left(c^2-1\right) m P'''_n +(n-6) (m (n-8)+4)  P'_n \nonumber \\
	            &+4 m (n-5) P''_{n-2} -4 c (m (n-8)+2) P''_n. \label{A.eq.3.21.Cox}
	        \end{align}
	        Multiplying it by $-2 \left(c^2-1\right) (3 m+2)$ and adding to  \eqref{A.eq.3.19.Cox} multiplied by $n-5$  gives us the following equation"
	        \begin{align*}
	            0=&-8 \left(c^2-1\right)^2 m (3 m+2) P'''_n \\
	            &+4 c \left(c^2-1\right) \left(m \left(c^2 (n-5) (m (n-1)+4)-m ((n-6) n+23)-4 (n-5)\right)+8\right) P''_n \\
	            &+c (n-5) (n-4) (m (n-6)+4) \left(c^2 (m (n-1)+4)-m (n+5)-8\right) P_n \\
	            &+4 c (n-5) \left(m \left(c^2 (n-5) (m (n-1)+4)+m (13-(n-6) n)-4 n+44\right)+8\right) P'_{n-2} \\
	            &+\left((n-5) \left(-4 c^4 (m (n-6)+2) (m (n-1)+4)+4 c^2 (m (n-6)+2) (m (n+5)+8)\right.\right.\\
	            &\left.\left.-8 (3 m+2) (m (n-3)+4)\right)-2 \left(c^2-1\right) (3 m+2) (n-6) (m (n-8)+4)\right)  P'_n \\
	            &+2 (3 m+2) (n-5) (n-2) (m (n-4)+4) P_{n-2}.
	        \end{align*}
	        
	       Multiplying now \eqref{A.eq.3.20.Cox} by $c \left(m \left(c^2 (n-5) (m (n-1)+4)+m (13-(n-6) n)-4 n+44\right)+8\right)$ and adding it to  the equation above multiplied by $-m$   gives the following equation:
	       
	        \begin{align*}
	            0=&-24 c \left(c^2-1\right) m^2 P''_n -4 \left(c^2-1\right)^2 m^2 P'''_n \\
	            &+\left(c^2 (m (3 n (m (n-6)+4)-40)+16)\right.\\
	            &\left.+m (-3 m ((n-6) n+4)-12 n+56)\right)  P'_n \\
	            &+m (n-5) (n-2) (m (n-4)+4) P_{n-2} \\
	            &-c (n-4) (m (n-6)+4) (m (n-2)+2) P_n.
	        \end{align*}
	        If we differentiate this equation, multiply
	        by $-4$ and add to \eqref{A.eq.3.20.Cox} multiplied by $(n-2) (m n-4m+4)$ we will come
	          to  the desired  linear differential equation of  order $4$. This proves the theorem.
	      
	     \end{proof}

	        Denote $v:=\ln z$ and, consequently, $\frac{\partial}{\partial v}=z\frac{\partial }{\partial z}$.
	        \begin{corollary}\label{cor:P_4PDE} 
	        Generating function $P_{-4}=P_{-4}(c,z)$ satisfies the following linear PDE of fourth order:
	            \begin{multline}
	0= m^2\left[4(1-c^2)\frac{\partial^2 }{\partial c^2}-12 c\frac{\partial }{\partial c}+\frac{\partial^2 }{\partial v^2} +2 (\frac{2}{m}-3)\frac{\partial }{\partial v}+4(1-\frac{2}{m})\right]^2 P_{-4}\\
	+16(m^2-8m)\frac{\partial^2 P_{-4}}{\partial c^2}-16(2-m)^2(c\frac{\partial }{\partial c}+1)^2P_{-4}, \label{eqn:P_4PDE}
	\end{multline}
	 whith $(c,z)\in (-1,1)\times (0,1)$ and with the following boundary conditions:
	\begin{equation}\label{eqn:P_4PDE_BC}
	    P_{-4}(c,0)=1, \frac{\partial P_{-4}}{\partial z}(c,0)=\frac{\partial^2 P_{-4}}{\partial z^2}(c,0)=\frac{\partial^3 P_{-4}}{\partial z^3}(c,0)=0.
	\end{equation}        
	        \end{corollary}
	        \begin{proof}
	        It follows from  \eqref{eqn:P_4IntegralFormula} that $P_{-4}$ is an analytic function represented by the series $P_{-4}(c,z)=\sum\limits_{n=0}^{\infty} P_{-4,n}(c)z^n$ (in some ball of small radius). Hence, by Theorem \ref{thm:P_4ODE_System} and some elementary algebraic calculations, $P_{-4}$ satisfies the linear PDE \eqref{eqn:P_4PDE} in the ball. By the uniqueness of analytic continuation, it satisfies the PDE \eqref{eqn:P_4PDE} in the domain $(-1,1)\times (0,1)$. The boundary conditions \eqref{eqn:P_4PDE_BC} immediately follow from the definition of $P_{-4}$.
	        \end{proof}
	        \begin{remark}
	Boundary conditions for $P_{-4}$ on other parts of the boundary of the domain $[-1,1]\times [0,1]$ can be explicitly calculated using formula \eqref{eqn:P_4IntegralFormula}.
	        \end{remark}
	    \subsection{Superelliptic type 2}

	        Now we consider the second family of polynomials $P_{-2,n}$.
	       We have
	    
	        \begin{align*}
	            P_{-2}(c,z)=&\sum _{k=0}^{\infty }P_{-2,k-4}(c) z^{k}=
	            \\
	            =&z^{3-\frac{4}{m}} \left(-2 c z^2+z^4+1\right)^{1/m} \lim_{\epsilon\to 0}\left[\int_{\epsilon}^{z} \frac{(4-m) w^{\frac{4}{m}-2}}{m \left(-2 c w^2+w^4+1\right)^{\frac{1}{m}+1}} \, dw +\epsilon^{4/m-1}\right]= \\
	            =&z^2-\frac{2 c (m-2)}{m+4}z^4+\frac{z^6 \left(m (m+4)-4 c^2 \left(m^2-4\right)\right)}{(m+4) (3 m+4)}\\
	            &-\frac{4 z^8 \left(c (m+2) \left(2 c^2 (m-2) (3 m+2)-3 m^2\right)\right)}{(m+4) (3 m+4) (5 m+4)}+\textrm{O}(z^9),
	        \end{align*}
	        and the polynomials $P_{-2,n}(c)$ satisfy the following recursion:
	        \begin{equation}
	            0=2 c (km-m+2) P_{-2,k+2}(c)-(k m+m+4) P_{-2,k+4}(c)-(km-3m) P_{-2,k}(c).
	        \end{equation}
	        
	        In particular,
	        the first nonzero polynomials of the sequence are:
	        \begin{align*}
	            P_{-2,2}(c)=&1, \\
	            P_{-2,4}(c)=&-\frac{2 c (m-2)}{m+4}, \\
	            P_{-2,6}(c)=&\frac{m (m+4)-4 c^2 \left(m^2-4\right)}{(m+4) (3 m+4)}, \\
	            P_{-2,8}(c)=&-\frac{4 c (m+2) \left(2 c^2 (m-2) (3 m+2)-3 m^2\right)}{(m+4) (3 m+4) (5 m+4)}.
	        \end{align*}
	        \begin{theorem}
	            The polynomials $P_n=P_{-2,n}(c)$ satisfy the following fourth order linear differential equation:
	            \begin{align}\label{equacaodiferencial2}
	                0=&16 \left(c^2-1\right)^2 m^2 P_n{}^{(iv)}+160 c \left(c^2-1\right) m^2 P'''_n\nonumber\\
	                &+\left[8 m (m (n^2-6n-2)+4 (n-8))\right.\nonumber\\
	                &\left.-8 c^2 \left(m^2 ((n-6) n-32)+4 m (n-6)+8\right)\right] P''_n\nonumber\\
	                &-24 c \left(m^2 (n-4) (n-2)+4 m (n-6)+8\right) P'_n\nonumber\\
	                &+(n^2-4) (m n-8m+4) (m n-4m+4) P_n.
	            \end{align}
	        \end{theorem}

	       \begin{corollary}\label{cor:P_2PDE}
	        Generating function $P_{-2}=P_{-2}(c,z)$ satisfies the following linear PDE of fourth order:
	            \begin{multline}
	0=m^2\left[4(1-c^2)\frac{\partial^2 }{\partial c^2}-12 c\frac{\partial }{\partial c}+\frac{\partial^2 }{\partial v^2} +2 (\frac{2}{m}-3)\frac{\partial }{\partial v}-4\right]^2 P_{-2}\\
	+144 m^2\frac{\partial^2 P_{-2}}{\partial c^2}-16(3m-2)^2(c\frac{\partial }{\partial c}+1)^2 P_{-2}, \label{eqn:P_2PDE}
	\end{multline}
	(where $v=\ln z$, $\frac{\partial}{\partial v}=z\frac{\partial }{\partial z}$) with $(c,z)\in (-1,1)\times (0,1)$ and with the following
	 boundary conditions:
	\begin{equation}\label{eqn:P_2PDE_BC}
	    \frac{\partial^2 P_{-2}}{\partial z^2}(c,0)=1, P_{-2}(c,0)=\frac{\partial P_{-2}}{\partial z}(c,0)=\frac{\partial^3 P_{-2}}{\partial z^3}(c,0)=0.
	\end{equation}        
	        \end{corollary}
	        \begin{proof}
	          The proof is analogous to the proof of Corollary \ref{cor:P_4PDE}.
	        \end{proof}
	        
	         \section{Uniqueness of  solutions}
	         Let us check now if the differential equations \eqref{eqdiferencial1} and \eqref{equacaodiferencial2} 
	         have other polynomial solutions. The argument below was suggested by A.Shyndyapin. 
	         
	         Clearly, if $\{P_n\}$ is a family of polynomials satisfying \eqref{eqdiferencial1} (respectively, \eqref{equacaodiferencial2}),
	         then for any choice of nonzero scalars $\lambda_n\in \Cx$, the family $\{\lambda_n P_n\}$ also satisfies \eqref{eqdiferencial1} (respectively, \eqref{equacaodiferencial2}). We will call such families of polynomials equivalent. Also, we will ignore the families of polynomials obtained by zeroing some members of the family. 
	          
	            Consider first the equation \eqref{eqdiferencial1} and assume that   
	            $$Q_n=\sum\limits_{i=0}^{r(n)} \alpha_i(n) c^i$$
	             is  its polynomial solution. We will simply write $r$ for $r(n)$ and $\alpha_i$ for $\alpha_i(n)$.
	             After the substitution in \eqref{eqdiferencial1} we obtain
	            
	            \begin{align*}
	                \sum _{i=0}^r &((2 i-n+4) (2 i+n) (m (2 i-n+6)-4) (m (2 i+n-2)+4))\alpha_ic^{i}\\
	                &+(-8 (i-1) i m \left(m \left(4 i^2-(n-6) n-6\right)-4 (n-6)\right))\alpha_ic^{i-2}\\
	                &+(16 (i-3) (i-2) (i-1) i m^2)\alpha_ic^{i-4}=0,
	            \end{align*}
	            which is equivalent to the following system of linear equations:
	            \begin{equation}
	                \label{system1}
	                \begin{array}{c}
	                    [(2 i-n+4) (2 i+n) (m (2 i-n+6)-4) (m (2 i+n-2)+4)]\alpha_i\\
	                    +[-8 (i+1) (i+2) m (m (4 i (i+4)-(n-6) n+10)-4 (n-6))]\alpha_{i+2}\\
	                    +[16 (i+1) (i+2) (i+3) (i+4) m^2]]\alpha_{i+4}=0,\,\,\,\,\, i=0,1, \ldots, r.\\
	                \end{array}
	            \end{equation}
	            Here $\alpha_s$ is assumed to be zero if $s>r$. The coefficient matrix of the system is upper triangular 
	            with 
	            $$(2 i-n+4) (2 i+n) (m (2 i-n+6)-4) (m (2 i+n-2)+4), \,\, i=0, \ldots, r$$ 
	            on the diagonal. Hence, for each $n$, the solution of the system is unique and trivial if $(2 i-n+4) (m (2 i-n+6)-4)\neq 0$ for all $i=0,1, \ldots, r$. Suppose $(2 i-n+4) (m (2 i-n+6)-4)= 0$. First consider the case $m>4$. Then  we must have $2i-n+4=0$.  From \eqref{system1} we get that  $Q_{2k+1}=0$ for all $k>0$, and $Q_{2k}$ is defined  uniquely up to a scalar for each $k>1$. The degree of 
	      $Q_{2k}$ is $\geq k-2$.   Now, if $m=4$ then either   $i=\frac{n}{2}-2$ or   $i=\frac{n+1}{2}-3$. Therefore, 
	       $Q_{2k}$ is defined uniquely up to a scalar   as a   polynomial of degree $\geq k-2$, $Q_1=Q_3=0$ and  $Q_{2k+1}$ is defined uniquely up to a scalar   as a   polynomial of degree $\geq k-2$ for each $k\geq 2$.
	            
	            
	            

	            Consider now the equation\eqref{equacaodiferencial2}. Arguing as above we come a linear system on the coefficients of the polynomial solution $Q_n$. The matrix of this system is again upper triangular with diagonal entries
	            $$\left(4 (i+1)^2-n^2\right) (m (2 i-n+8)-4) (m (2 i+n-4)+4),\,\, i=0, \ldots, r.$$ 
	          Assume that  $\left(4 (i+1)^2-n^2\right) (m (2 i-n+8)-4) (m (2 i+n-4)+4)=0$. 
	            If $m>4$, then   $Q_{2k+1}=0$ for all $k> 0$, and $Q_{2k}$ is defined  uniquely up to a scalar for each $k> 1$, while $Q_2=0$. The degree of 
	      $Q_{2k}$ is $\geq k-2$.   Suppose now that $m=4$ and $n>1$.  Then either   $i=\frac{n}{2}-1$ or   $i=\frac{n+1}{2}-3$. Therefore, 
	       $Q_{2k}$ is defined uniquely up to a scalar   as a   polynomial of degree $\geq k-2$, $Q_3=Q_5=0$
	       and  $Q_{2k+1}$ is defined uniquely up to a scalar   as a   polynomial of degree $\geq k-2$ for each $k\geq 3$. When $m=4$, $n=1$ and $i=1$ we also get $0$. Hence, $Q_1$   is defined uniquely up to a scalar. 
	     
	       
	            We proved the following theorem.
	
	            \begin{theorem} For any $m\geq 4$, 
	                the differential equations \eqref{eqdiferencial1} and \eqref{equacaodiferencial2} 
	               have  unique  polynomial solutions up to equivalence.
	            \end{theorem}
	    
	\section{Associated ultraspherical polynomials}
	
	    After shifting the indices back by $4$ we obtain that both families of the polynomials $P_{-4,n}(c)$ and $P_{-2,n}(c)$ satisfy the recurrence relation
	    \begin{equation}
	        (k m+m+4) P_k(c)=2 c (km-m+2)P_{k-2}(c)-(k-3) m P_{k-4}(c),
	    \end{equation}
	    where $P_k(c)\in \{P_{-4,k-4}(c), P_{-2,k-4}(c)\}$. 
	    Note that all odd polynomials are zero. Set $k=2(n+1)$ and $q_s=q_s(c):=P_{2s}(c)$. Then we have
	    \begin{equation}
	        2 c (2 m n+m+2) q_n=(2m n+3m+4) q_{n+1}+m (2 n-1) q_{n-1}. 
	        \label{A.eq.4.2.Cox}
	    \end{equation}
	    
	    For example, we have:  $q_{0} =1, 
	            q_1=\frac{6 c m (m+2) }{(m+4) (3 m+4)}$, 
	$$
	  q_2=\frac{3 m (3 m+2) \left(4 c^2 (m+2) -2m \right)}{(m+4) (3 m+4) (5 m+4)}, \,\,\,
	            q_3=\frac{12cm (3 m+2)(5 m+2)  \left(2 c^2 (m+2) -6m \right)}{(m+4) (3 m+4) (5 m+4) (7 m+4)}.$$

	    We will show that $\{q_n\}$ is a special case of the associated ultraspherical polynomials.
	        
	    Let $c$ and $\nu$ are two complex numbers. Let $C_n^{(\nu)}(x;c)$ denotes the family of \textit{associated ultraspherical polynomials} with initial conditions $C_{-1}^{(\nu)}(x;c)=0$ and $C_{0}^{(\nu)}(x;c)=1$ (cf. \cite{BustozIsmail1982}, p.~729). Then they  satisfy the following difference equation:
	    \begin{equation}
	        2x(n+\nu+c)C_{n}^{(\nu)}(x;c)=(n+c+1)C_{n+1}^{(\nu)}(x;c)+(2\nu+n+c-1)C_{n-1}^{(\nu)}(x;c).
	    \end{equation}
	    Choosing $c=\frac{1}{2}+\frac{2}{m}$ and $\nu=-\frac{1}{m}$ this equation becomes the  recurrence relation  \eqref{A.eq.4.2.Cox}. Hence, we conclude that polynomials $P_{-4,n}$, is a special case of associated ultraspherical polynomials. We immediately have
	    
	    \begin{corollary}
	        Polynomials $P_{-4,n}$ are non-classical orthogonal polynomials.
	    \end{corollary}
	    \section{Orthogonality of $P_{-2,n}$}
	    
	    Polynomials $P_{-2,n}$ satisfy the same recurrence relation as polynomials $P_{-4,n}$ but have different initial conditions. Hence, we do not know immediately whether they are associated ultraspherical polynomials and whether they are orthogonal. We will give an independent proof of the orthogonality of these polynomials.
	     Set $q_s(c)=P_{-2,2s}(c)$. In particular, we have
	   $$
	        q_{1}=1\,\,\, q_{2}=-\frac{2 c (m-2)}{m+4}, \,\,\,
	        q_{3}= \frac{m (m+4)-4 c^2 \left(m^2-4\right)}{(m+4) (3 m+4)}.
	    $$
	   Now we set
	 $  \overline{q}_n:=q_{n+1}$,  $n\geq -1$.
	    Then polynomials $\overline{q}_{n}$ have degree $n$, and  \eqref{A.eq.4.2.Cox} becomes
	    \begin{equation}
	        2  (2m n+3m+2)c \overline{q}_{n}=(2 m n+m) \overline{q}_{n-1}+(2m n+5m+4) \overline{q}_{n+1}.
	    \end{equation}
	    
	    We recall the following result.
	    \begin{theorem}[\cite{Cox2013DJKMPolynomials}, Theorem 5.0.4] \label{a.cox.5.0.4} Let $\{p_n, n \geq 0\}$ be a sequence of polynomials, where $p_n$ has degree $n$ for any $n$, satisfying the following recursion
	        \begin{align}
	            \chi p_n = a_{n+1}p_{n+1}+b_np_n+c_{n-1}p_{n-1}, && n=0,1,2,\dots
	        \end{align}
	        for some complex numbers $a_{n+1}$, $b_n$, $c_{n-1}$, $n=0, 1, 2, \dots$ and $p_{-1}=0$. Then $\{p_n, n\geq 0\}$ is a sequence of orthogonal polynomials with respect to some (unique) weight measure if and only if $b_n\in \mathbb{R}$ and $c_n=\overline{a}_n+1\neq 0$ for all $n\geq 0.$
	    \end{theorem}
	    We apply the theorem in the case when $$a_{n+1}=\frac{(2m n+5m+4) }{2  (2m n+3m+2)}, \,\,\, b_{n}=0, \,\,\, c_{n-1}=\frac{(2 m n+m)}{2  (2m n+3m+2)}.$$
	    
	    \begin{theorem}
	         The polynomials $P_{-2,n}(c)$ are orthogonal with respect to some weight function.
	    \end{theorem}
	    \begin{proof}
	        It is sufficient to check that there exists a family of orthonormal polynomials $f_n$ and constants $\lambda_n$ such that $q_n=\lambda_n f_n$ for all $n$.
	        
	        We have the following recursion for $f_n$:
	    \begin{equation*}
	        f_n=\frac{(2 m n+m)\lambda_{n-1}}{2  (2m  n+3m+2)c \lambda_{n} } f_{n-1}+\frac{(2m  n+5m+4) \lambda_{n+1}}{2  (2m n+3m+2)c \lambda_{n} }f_{n+1},
	    \end{equation*}
	    for $n\geq 1$.
	    
	    Set
	    \begin{align*}
	        A_{n+1} = \frac{(2m  n+5m+4) \lambda_{n+1}}{2  (2m n+3m+2) \lambda_{n} }& & C_{n-1}=\frac{(2 m n+m)\lambda_{n-1}}{2  (2m  n+3m+2) \lambda_{n} }
	    \end{align*}
	    Then $A_{n}=C_{n-1}$ if and only if 
	    \begin{equation*}
	        \lambda_n^2=\frac{(2m  n+m+2) (2 m n+m)}{(2m  n+3m+4)(2m  n+3m+2)  }\lambda_{n-1}^2.
	    \end{equation*}
	    Taking $\lambda_0 = 1$ we can find a family of constants $\lambda_n$ satisfying this relation. 
	    \end{proof}

	   \section*{Acknowledgments}
	Vyacheslav Futorny was supported by NSFC of China, Grant 12350710178.
	    
	 \medskip
	 

	\bibliographystyle{elsarticle-num}

	\bigskip
	
	  Felipe ~Albino dos Santos: University of S\~ao Paulo and Mackenzie Presbyterian University, S\~ao Paulo, Brazil\par\nopagebreak
	  \textit{email}: falbinosantos@gmail.com
	
	  Vyacheslav ~Futorny: Shenzhen International Center for Mathematics, SUSTech, China and University of S\~ao Paulo, Brazil\par\nopagebreak
	  \textit{email}: vfutorny@gmail.com
	
	  Mikhail ~Neklyudov: Federal University of Amazonas, Manaus, Brazil\par\nopagebreak
	  \textit{email}: misha.neklyudov@gmail.com
\end{document}